\documentclass[11pt]{article}
\usepackage{amsmath,amssymb}

\newtheorem{propo}{{\bf Proposition}}[section]
\newtheorem{coro}[propo]{{\bf Corollary}}
\newtheorem{lemma}[propo]{{\bf Lemma}} \newtheorem{theor}[propo]{{\bf
Theorem}} \newtheorem{ex}{{\sc Example}}[section]
\newtheorem{definition}{\bf Definition}
\newenvironment{proof}{{\bf Proof.}}{$\Box$}

\def\N{{\mathbb N}}
\begin{document}
\vspace*{1.0in}

\begin{center} Leibniz algebras in which all centralisers of nonzero elements are zero algebras 

\end{center}
\bigskip

\centerline {David A. Towers} \centerline {Department of
Mathematics, Lancaster University} \centerline {Lancaster LA1 4YF,
England}
\bigskip

\begin{abstract} This paper is concerned with generalising the results for Lie $CT$-algebras to Leibniz algebras. In some cases our results give a generalisation even for the case of a Lie algebra. Results on $A$-algebras are used to show every Leibniz $CT$-algebra over an algebraically closed field of characteristic different from 2,3 is solvable or is isomorphic to $sl_2(F)$. A characterisation is then given for solvable Leibniz $CT$-algebras. It is also shown that the class of solvable Leibniz $CT$-algebras is factor closed. 
\end{abstract}


\section{Introduction}
An algebra $L$ over a field $F$ is called a {\it Leibniz algebra} if, for every $x,y,z \in L$, we have
\[  [x,[y,z]]=[[x,y],z]-[[x,z],y].
\]
In other words, the right multiplication operator $R_x : L \rightarrow L : y\mapsto [y,x]$ is a derivation of $L$. As a result such algebras are sometimes called {\it right} Leibniz algebras, and there is a corresponding notion of {\it left} Leibniz algebras, which satisfy
\[  [x,[y,z]]=[[x,y],z]+[y,[x,z]].
\]
Clearly, the opposite of a right (left) Leibniz algebra is a left (right) Leibniz algebra, so, in most situations, it does not matter which definition we use. 
\par
 
Every Lie algebra is a Leibniz algebra and every Leibniz algebra satisfying $[x,x]=0$ for every element is a Lie algebra. They were introduced in 1965 by Bloh (\cite{bloh}) who called them $D$-algebras, though they attracted more widespread interest, and acquired their current name, through work by Loday and Pirashvili ({\cite{loday1}, \cite{loday2}). They have natural connections to a variety of areas, including algebraic $K$-theory, classical algebraic topology, differential geometry, homological algebra, loop spaces, noncommutative geometry and physics. A number of structural results have been obtained as analogues of corresponding results in Lie algebras.
\par

The {\it Leibniz kernel} is the set $I=$ span$\{x^2:x\in L\}$. Then $I$ is the smallest ideal of $L$ such that $L/I$ is a Lie algebra.
Also $[L,I]=0$.
\par

We define the following series:
\[ L^1=L,\;\; L^{k+1}=[L^k,L] \;\;\; (k\geq 1)\]
and 
 %\hbox{ and } 
\[L^{(0)}=L, \;\; L^{(k+1)}=[L^{(k)},L^{(k)}]\;\;\; (k\geq 0).
\]
Then $L$ is {\em nilpotent of class n} (resp. {\em solvable of derived length n}) if $L^{n+1}=0$ but $L^n\neq 0$ (resp. $ L^{(n)}=0$ but $L^{(n-1)}\neq 0$) for some $n \in \N$. It is straightforward to check that $L$ is nilpotent of class $n$ precisely when every product of $n+1$ elements of $L$ is zero, but some product of $n$ elements is non-zero. 
\par

The {\it centraliser} of an element $x$ of an algebra $A$ is $C_A(x)=\{a\in A \mid ax=xa=0\}$; if the algebra $A$ is clear we will simply write $C(x)$. An algebra $A$ is called {\it centraliser transitive, or a $CT$-algebra} if $x\in C(y)$ and $y\in C(z)$ imply that $x\in C(z)$. Such algebras, where $A$ is a Lie algebra, have been studied by Klep and Moravec (\cite{km}) in 2010, by Arzhantsev, Makedonskii, and Petravchuk (\cite{km}) in 2011, and by Gorbatsevic (\cite{gorb}) in 2016. In the last two of these references they are shown to have applications to the classification of finite-dimensional subalgebras in polynomial Lie algebras of rank one and to Lie algebras of vector fields whose orbits are one-dimensional. A similar notion for groups was defined and studied by Weisner {\cite{weis}) in 1925.  Finite nonabelian simple $CT$-groups had been classified by Suzuki (\cite{suz})  in 1957. He proved that every finite nonabelian simple $CT$-group is isomorphic to some $PSL(2,2^f )$, where $f>1$. Suzuki's result is considered to have been one of the keystones in the proof of the Odd Order Theorem by Feit and Thompson ({\cite{ft}).
\par

Note that $x\in C(y)$ if and only if $y\in C(x)$. It is clear that $CT$-algebras are subalgebra closed. The {\it centre} of an algebra $A$ is $Z(A)=\{z\in A \mid za=az=0 \hbox{ for all } a\in A\}$. We call $A$ a {\it zero algebra} if $A=Z(A)$. The following lemma makes clear the title of the paper.

\begin{lemma}\label{1} The algebra $A$ is a $CT$-algebra if and only if $C(x)$ is a zero algebra for all $0\neq x\in L$.
\end{lemma}
\begin{proof} Let $A$ be a $CT$-algebra and let $y,z\in C(x)$. Then $y\in C(x)$ and $x\in C(z)$, so $y\in C(z)$; that is $yz=zy=0$.
\par

Conversely, suppose that $C(x)$ is a zero algebra for all $x\neq 0$, and let $x\in C(y)$ and $y\in C(z)$. Then $x,z\in C(y)$, so $xz=zx=0$, whence $x\in C(z)$.
\end{proof}

\begin{lemma}\label{z} If $A$ is a $CT$-algebra and $Z(A)\neq 0$, then $A$ is a zero algebra.
\end{lemma}
\begin{proof} Let $0\neq x\in Z(A)$. Then $A=C(x)$ is a zero algebra.
\end{proof}
\medskip

We call an algebra $A$ an {\it $A$-algebra} if every nilpotent subalgebra of $A$ is a zero algebra. Then $CT$-algebras are a subclass of the class of $A$-algebras. 

\begin{lemma}\label{A} Every $CT$-algebra is an $A$-algebra.
\end{lemma}
\begin{proof} Let $N$ be a nilpotent $CT$-algebra. Then $Z(N)\neq 0$, so the result follows from Lemma \ref{z}..
\end{proof}
\medskip

The next section is concerned with generalising the results for Lie $CT$-algebras to Leibniz algebras. In some cases our results give a generalisation even for the case of a Lie algebra. Leibniz $A$-algebras were studied by Towers in \cite{tow}. The results proved there are used to show every Leibniz $CT$-algebra over an algebraically closed field of characteristic different from 2,3 is solvable or is isomorphic to $sl_2(F)$; the previously recorded result for Lie algebras asumed characteristic zero. A characterisation is then given for solvable Leibniz $CT$-algebras. It is also shown that the class of solvable Leibniz $CT$-algebras is factor closed. 
\par

From now on, $L$ will denote a finite-dimensional Leibniz algebra over a general field $F$ (unless specified otherwise). The notation $\oplus$ will denote an algebra direct sum, whereas $\dot{+}$ will denote a direct sum of the underlying vector space structure alone.

\section{Main results}
The following lemma will prove useful.

\begin{lemma}\label{s} Let $L=A\dot{+}Fx$ be a Leibniz $CT$-algebra over a field $F$, where $A$ is a zero ideal of $L$ and $x^2\in A$. Then, either $L$ is a zero algebra or $R_x|_A$ has no zero eigenvalue in $F$ and $Ax=A$. 
\end{lemma}
\begin{proof} Suppose that $R_x|_A$ has a zero eigenvalue in $F$.  Let $0\neq y\in A$ be a corresponding eigenvector, so $yx=0$. If $xy=0$, we have that $L=C(y)$ is a zero algebra. So suppose that $xy\neq 0$. Then 
\[ (xy)x=x(yx)+x^2y=0 \hbox{ and } x(xy)=x^2y-(xy)x=0.
\] Thus, $L=C(xy)$ is a zero algebra again. 
\par

Suppose next that $R_x|_A$ has no zero eigenvalue in $F$. Let $L=L_0\dot{+}L_1$ be the Fitting decomposition of $L$ relative to $R_x$. Then $L_1\subseteq A$. Suppose that $B=L_0\cap A\neq 0$ and that $R_x^n(B)=0$, $R_x^{n-1}(B)\neq 0$. Let $0\neq y\in R_x^{n-1}(B)$. Then $yx=0$ and $R_x|_A$ has a zero eigenvalue in $F$, a contradiction. It follows that $B=0$ and so $A=L_1$ and $Ax=A$. 
\end{proof}
\medskip

{\it Cyclic} Leibniz algebras, $L$, are generated by a single element. In this case $L$ has a basis $a,a^2, \ldots, a^n (n > 1)$ and product $a^na=\alpha_2a^2+ \ldots + \alpha_na^n$.  Let $T$ be the matrix for $R_a$ with respect to the above basis. Then $T$ is the
companion matrix for $p(x) =  x^n - \alpha_n x^{n-1} - \ldots - \alpha_2 x= p_1(x)^{n_1} \ldots p_r(x)^{n_r}$, where the
$p_j$ are the distinct irreducible factors of $p(x)$. Then we have the following result.

\begin{theor} Let $L$ be a cyclic Leibniz algebra. Then the following are equivalent:
\begin{itemize}
\item[(i)] $L$ is a $CT$-algebra;
\item[(ii)] $L$ is an $A$-algebra; and
\item[(iii)] $\alpha_2 \neq 0$, and then $L=L^2\dot{+} F(a^n-\alpha_n a^{n-1} - \dots - \alpha_2 a)$ and we can take $p_1(x)^{n_1}=x$. 
\end{itemize}
\end{theor}
\begin{proof} The equivalence of (ii) and (iii) is given by \cite[Theorem 12]{tow}. Lemma \ref{A} gives that (i) implies (ii), so it simply remains to show that (iii) implies (i).
\par

So suppose that $L$ is as described in (iii). Put $b=a^n-\alpha_n a^{n-1} - \dots - \alpha_2 a$. Then it is easy to check that $b^2=0$ and $R_b|_{L^2}$ has no zero eigenvalue in $F$. Let $x=n+\lambda b$ where $n\in L^2$ and $\lambda\in F$. Now $bL^2=bI=0$, so straightforward calculations show that
\[   C(x)=
\left\{ \begin{array}{lll}
L^2 & \hbox{if} & \lambda=0, n\neq 0 \\ 
Fb & \hbox{if} &  \lambda\neq 0, n=0 \\
0 &  \hbox{if} & \lambda\neq 0, n\neq 0 
\end{array} \right\}
\] Hence $L$ is a $CT$-algebra 
\end{proof}
\medskip

{\bf Note} that the above result shows that we may not have $xA=A$ in Lemma \ref{s}. For, if $L$ is a cyclic Leibniz $A$-algebra, then $Ix=I$, but $xI=0$.

\begin{lemma} If $L$ is a Leibniz algebra and $N$ is a zero ideal of $L$, we can consider $N$ as a right $L/N$-module under the action $n(x+N)=nx$ for all $x\in L$, $n\in N$. Then, each element $R_{x+N}$ with $x\notin N$ has no zero eigenvalue under this action if and only if $C_L(n)\subseteq N$ for all $n\in N$. 
\end{lemma}
\begin{proof} Suppose $R_{x+N}$ has a zero eigenvalue under the action, where $x\notin N$. Then, there exists $0\neq n\in N$ such that $nx=0$. If $xn=0$, then $x\in C_L(n)\setminus N$. If $xn\neq 0$, then $x\in C_L(xn)\setminus N$, as in Lemma \ref{s} above.
\par

Conversely, if $x\in C_L(n)\setminus N$ for some $n\in N$, we have $nx=0$ and $R_{x+N}$ has a zero eigenvalue.
\end{proof}

\begin{propo}\label{factor} Let $L$ be a Leibniz algebra and $N$ be a zero ideal of $L$ such that $L/N$ is a $CT$-algebra and $C_L(n)\subseteq N$ for all $0\neq n\in N$. Then $L$ is a $CT$-algebra. 
\end{propo}
\begin{proof} Let $0\neq x\in L$ and take $y,z\in C_L(x)$, so $xy=yx=xz=zx=0$. If $y\in N$, then $x\in C_L(y)\subseteq N$. Hence $y,z\in C_L(x)\subseteq N$, giving $yz=0$.
\par

So assume that $y\notin N$. Then $yz\in N$, by the hypothesis. Now
\[ x(yz)=(xy)z-(xz)y=0 \hbox{ and } (yz)x=y(zx)+(yx)z=0,
\] so $x\in C_L(yz)\subseteq N$. Thus $C_L(x)\subseteq N$ and $yz=0$. The result follows.
\end{proof}
\medskip

{\bf Note} that If $L$ is a nilpotent cyclic Leibniz algebra of dimension greater than $1$, then $L/I$ is a Lie $CT$ algebra, but $L$ is not a $CT$-algebra. Next we have the two main classification results.

\begin{theor}\label{main} Let $L$ be a Leibniz $CT$-algebra over an algebraically closed field $F$ of characteristic $\neq 2,3$. Then $L$ is solvable or is isomorphic to $sl_2(F)$.
\end{theor}
\begin{proof} Since $F$ has cohomological dimension $0$, $L=R\dot{+}S$, where $R$ is the radical of $L$ and $S$ is a direct sum of ideals isomorphic to $sl_2(F)$, by Lemma \ref{A} and \cite[Theorem 2]{tow}. Suppose that $R\neq 0$. If $S\neq 0$, it is clear from Lemma \ref{1} that $S\cong sl_2(F)$, and there is an element $x\in S$ such that $R_x$ is nilpotent. Then $R_x$ has a zero eigenvalue and $A=N\dot{+} Fx$, where $N$ is the nilradical of $L$, must be a zero algebra, by Lemma \ref{s}. Thus, $x\in C_L(N)$. But $C_L(N)$ is an ideal of $L$, and so $S\subseteq C_L(N)$. Pick $0\neq y\in N$. Then $S\subset C(y)$, which is a zero algebra. Hence $S=0$.
\end{proof}

\begin{theor}\label{solv} Let $L$ be a solvable Leibniz $CT$-algebra of derived length $n+1$. Then 
\begin{itemize}
\item[(i)] $L=A_n\dot{+} A_{n-1}\dot{+} \ldots \dot{+} A_0$, where $A_i$ is an a zero subalgebra of $L$ and $L^{(i)}=A_n\dot{+} A_{n-1}\dot{+} \ldots \dot{+} A_i$ for $0\leq i\leq n$;
\item[(ii)] $L$ splits over the nilradical $N$, which equals $L^{(n)}$; and
\item[(iii)] for every element $x\in L\setminus N$, $R_x|_N$ has no zero eigenvalue in $F$, and $Nx=N$.
\end{itemize}
\end{theor}
\begin{proof} (i) We have that $L=A_n\dot{+} A_{n-1}\dot{+} \ldots \dot{+} A_0$ and $L^{(i)}=A_n\dot{+} A_{n-1}\dot{+} \ldots \dot{+} A_i$, by \cite[Corollary 1]{tow}. 
\par

\noindent (ii) Also, $N=A_n\oplus (N\cap A_{n-1}\oplus \ldots \oplus N\cap A_0)$ and $Z(L^{(i)})=N\cap A_i$ for each $0\leq i\leq n$, by \cite[Theorem 5]{tow}. 
Suppose that $N\cap A_i\neq 0$ for some $0\leq i\leq n-1$. Let $0\neq x\in N\cap A_i$. Then $C_L(x)\supseteq L^{(i)}$. But $L^{(i)}$ is not a zero algebra for $0\leq i\leq n-1$, so we have a contradiction.
\par

\noindent (iii) Let $x\in L\setminus N$. Then $L(x)=N\dot{+}Fx$ satisfies the hypothesis of Lemma \ref{s}, since $x^2\in I\subseteq N$. If $L(x)$ is a zero algebra, then $x\in C_L(N)=N$, by \cite[Lemma 7]{tow}, a contradiction. Hence  $R_x|_N$ has no zero eigenvalue in $F$, and $Nx=N$, by Lemma \ref{s}.
\end{proof}

\begin{definition} A Leibniz algebra $L$ is called {\bf completely solvable} if $L^2$ is nilpotent. 
\end{definition}

Over a field of characteristic zero, every solvable Leibniz algebra is completely solvable, and so every solvable Leibniz $CT$-algebra has derived length at most $2$. However, this is not the case over fields of positive characteristic, even for Lie algebras, as the following example, which is taken from \cite[pages 52, 53]{jac}, shows. 

\begin{ex}\label{jac}
Let
$$ e = \left[ \begin{array}{rrrrrrr}
0 & 1 & 0 & . & . & . & 0\\
0 & 0 & 1 & 0 & . & . & 0\\
\vdots &  &  &  &  &  & \vdots \\
0 & . & . & . & . & 0 & 1 \\
1 & 0 & . & . & . & . & 0
\end{array} \right], \hspace{.2cm}
f = \left[ \begin{array}{rrrrr}
0 & 0 & 0 & \ldots & 0 \\
0 & 1 & 0 & \ldots & 0 \\
0 & 0 & 2 & \ldots & 0 \\
\vdots &  &  &  & \vdots \\
0 & 0 & 0 & \ldots & p-1
\end{array} \right],
$$ 
let $F$ be a field of prime characteristic $p$ and put $L = Fe + Ff + F^p$ with product $[a + {\bf x}, b + {\bf y}] = [a,b] + ({\bf x}b - {\bf y}a)$ for all $a, b \in Fe + Ff$, ${\bf x}, {\bf y} \in F^p$. 
\end{ex}

Then, straightforward calculations show that
\[   C(\alpha e+\beta f+{\bf x})=
\left\{ \begin{array}{lll}
Ff+F{\bf x_1} & \hbox{if} & \alpha=\beta=0, {\bf x}={\bf x_1} \\ 
 F^p & \hbox{if} &  \alpha=\beta=0, {\bf x}\neq {\bf x_1} \\
F(\alpha e+\beta f+{\bf x}) &  \hbox{if} & \alpha\neq 0 \\
F(\beta f+{\bf x})+F{\bf x_1} &  \hbox{if} & \alpha =0, \beta \neq 0,
\end{array} \right\}
\] where ${\bf x_1}=(1,0, \ldots, 0)$, and all of these are zero algebras.

\begin{theor}\label{solv1} Let $L$ be a solvable Leibniz $CT$-algebra over an arbitrary algebraically closed field $F$. Then the nilradical of $L$ has codimension at most $2$ in $L$.
\end{theor}
\begin{proof} We have that $L=N\dot{+}A_1\dot{+} A_0$, by Theorem \ref{solv} and \cite[Theorem 14]{tow}. Suppose $A_1\neq 0$. Let $A$ be a minimal ideal of $B=N\dot{+}A_1$, inside $N$. Then $A$ is an irreducible $B$-bimodule, and so $A_1A=0$ or $ab=-ba$ for every $a\in A$, $b\in B$, by \cite[Lemma 1.9]{barnes}. In either case, $A$ is a minimal right $A_1$-module and $A=Fn$ is one-dimensional, by \cite[Lemma 5]{sf}. But $C_B(A)\subseteq N$ has dimension at least $\dim B-1$, so $\dim A_1\leq1$. The same argument shows that $\dim A_0\leq 1$, whence the result.
\end{proof}
\medskip

If the field $F$ in the above result has characteristic zero, then the codimension is at most one. However, over any field of characteristic $p>0$, the codimension can be two, even if $L$ is a Lie algebra, as Example \ref{jac} shows.
\par

If $L$ is a Leibniz algebra and $y\in L$, the {\it left centraliser of $y$ in $L$}, is $C_L^{\ell}(y)=\{x\in L \mid xy=0\}$. It is easy to check that this a subalgebra of $L$.

\begin{theor}\label{solv2} Let $L$ be a completely solvable Leibniz $CT$-algebra. Then, either $L$ is a zero algebra, or $L=N\dot{+}A_0$ where $N$ is the nilradical, $N^2=A_0^2=0$, $R_x|_N$ has no zero eigenvalue and $Nx=N$ for all $x\in A_0$. If $A_0$ and $A_0'$ are two complements to $N$ in $L$, then there exists $n\in N$ such that $(1+L_n)(A_0)=A_0'$.
\end{theor}
\begin{proof} We have that $L=N\dot{+}A_0$  where $N$ is the nilradical and $N^2=A_0^2=0$, by Theorem \ref{solv}. Suppose that $L$ is not a zero algebra. Then $A_0\neq 0$, $R_x|_N$ has no zero eigenvalue and $Nx=N$ for all $x\in A_0$, by Lemma \ref{s}.
\par

For every $y\in L\setminus N$ we have $N=Ny\subseteq Ly\subseteq L^2\subseteq N$, so $N=Ly$. Pick any $x\in L$. Then $xy\in N=Ny$, so there is an $n\in N$ such that $xy=ny$. Thus $(x-n)y=0$ and $x-n\in C_L^{\ell}(y)$. It follows that $L=N+C_L^{\ell}(y)$. Moreover, $N\cap C_L^{\ell}(y)=0$, since $y$ has no zero eigenvalue on $N$, so $C_L^{\ell}(y)$ is a complement to $N$ in $L$. 
\par

Let $A_0$ be any complement to $N$ in $L$. Then $y=n'+a$ for some $n'\in N=Ny$ and $a\in A_0$. Hence, there is an $n\in N$ such that $n'=-ny$ and $a=(1+L_n)(y)$. Put $\theta=1+L_n$, so $\theta(y)=a$ and $\theta(C_L^{\ell}(y))\subseteq C_L^{\ell}(a)$. But $A_0=C_L^{\ell}(a)$ from which the final claim follows.
\end{proof}
\medskip

{\bf Note} that the above result mirrors the corresponding result in Lie algebras (see \cite[Theorem 3]{km}). However, in that result, the complements are conjugate under the inner automorphism $1+ad\, n$. In our result, $1+L_a$ is not an automorphism, in general, as the following example shows.

\begin{ex}\label{2} Let $L$ be the two-dimensional solvable cyclic Leibniz algebra with basis $a,a^2$ and $a^2a=a^2$. It is easy to see that this is a $CT$-algebra. Put $\theta=1+L_{a^2}$. Then $\theta(a^2)=a^2$, whereas $\theta(a)\theta(a)=(a+a^2)(a+a^2)=a^2+a^2$.
\end{ex}

It would have been better if we could have taken $\theta=1+R_n$, which is an automorphism. However, if $N=I$ (as in Example \ref{2} above), $\theta(A_0)=A_0$. 
\par

Finally, we have that solvable Leibniz $CT$-algebras are factor-closed.

\begin{theor}\label{fac} Let $L$ be a solvable Leibniz $CT$-algebra and let $J$ be an ideal of $L$. Then $L/J$ is a $CT$-algebra.
\end{theor}
\begin{proof} Suppose that $J\not \subseteq N$ and let $x\in J\setminus N$. Then $N=Nx\subseteq J$, by Lemma \ref{s}, so $J\subset N$ or $N\subseteq J$. We use induction on the derived length $k$ of $L$. Suppose first that $k=2$, so $L$ is completely solvable. If $N\subseteq J$, we have that $L/N$ is a zero algebra and so $L/J$ is a $CT$-algebra.
\par

 So suppose now that $J\subset N$. A similar argument to that used in Lemma \ref{s} also shows that $J=Jx$ for all $x\notin N$. Let $x+J\in C_{L/J}(n+J)$ where $n\in N\setminus J$ and suppose that $x\notin N$. Then $nx\in J=Jx$, so $nx=jx$ for some $j\in J$. Thus $(n-j)x=0$. But now $n=j$, since $L$ is a $CT$-algebra, and this is a contradiction. It follows that $C_{L/J}(n+J)\subseteq N/J$. Now $L/N\cong (L/J)/(N/J)$ is a $CT$-algebra and hence so is $L/J$, by Proposition \ref{factor}.
\par

So suppose the result holds whenever $k\leq m$ and let $L$ have derived length $m+1$. Then $L=N\dot{+}B$ for some subalgebra $B$ of derived length $m$ of $L$, by Theorem \ref{solv} (ii). Now $L/N\cong B$ is a $CT$-algebra, by the inductive hypothesis. If $N\subseteq J$, then $L/J\cong (L/N)/(J/N)\cong B/B\cap J$, which is a $CT$-algebra, by the inductive hypothesis. If $J\subset N$, then $L/J$ is a $CT$-algebra as in paragraph two above. 
\end{proof}

\begin{coro} Let $L$ be any Leibniz $CT$-algebra over an algebraically closed field of characteristic $\neq 2,3$, and let $J$ be an ideal of $L$. Then $L/J$ is a $CT$-algebra.
\end{coro}
\begin{proof} This is immediate from Theorems \ref{main} and \ref {fac}.
\end{proof}

\end{document}